\documentclass[smallcondensed,runningheads,twosided]{svjour3}
\usepackage{amsmath,amssymb,mathtools,listings}
\usepackage{color}
\usepackage{graphicx}
\journalname{Submitted}
\smartqed

\newcommand*{\Ab}{{\mathbb A}}

\newcommand*{\Rb}{{\mathbb R}}
\newcommand*{\la}{{\langle}}
\newcommand*{\ra}{{\rangle}}
\newcommand*{\sh}{{\,\llcorner\!\llcorner\!\!\!\lrcorner\,}}

\newcommand*{\Sb}{{\mathbb S}}
\newcommand*{\Nb}{{\mathbb N}}

\newcommand*{\cL}{{\mathrm L}}

\usepackage{ytableau}
\definecolor{todocolor}{HTML}{fa73ff}
\usepackage[linecolor=white,backgroundcolor=white,bordercolor=white,textsize=tiny,textcolor=todocolor]{todonotes}

\usepackage{forest}

\forestset{
	decor/.style = {
		label/.expanded = {[inner sep = 0.2ex, font=\unexpanded{\tiny}]right:{$#1$}}
	},
	root/.style = {minimum size = 0.1ex},
	decorated/.style = {
		for tree = {
			circle, fill, inner sep = 0.3ex, minimum size = 1.ex,
			grow' = north, l = 0, l sep = 1.2ex, s sep = 0.7em,
			fit = tight, parent anchor = center, child anchor = center,
			delay = {decor/.option = content, content =}
		}
	},
	default preamble = {decorated, root},
	begin draw/.code={\begin{tikzpicture}[baseline={([yshift=-0.5ex]current bounding box.center)}]},
}

\begin{document}

\title{The Exponential Lie Series and a Chen--Strichartz Formula for L{\'e}vy Processes}

\author{Kurusch Ebrahimi--Fard$^{\mathbf{1}}$, Fr\'ed\'eric Patras$^{\mathbf{2}}$, Anke Wiese$^{\mathbf{3}}$}
\authorrunning{K.~Ebrahimi--Fard, F.~Patras, A.~Wiese}
\titlerunning{Exponential Lie Series and Chen--Strichartz Formula}
\institute{$^1$Department of Mathematical Sciences, NTNU, 
NO 7491 Trondheim, Norway.\\
$^2$Current address: Laboratoire Ypatia des Sciences Mathématiques, Roma, Italia. Permanent address: Univ.~C\^ote d'Azur et CNRS, UMR 7351, Parc Valrose, 06108 Nice Cedex 02, France.\\
$^3$Maxwell Institute for Mathematical Sciences
and School of Mathematical and Computer Sciences,
Heriot-Watt University, Edinburgh EH14 4AS, UK}%\\

\date{Received: \today}

\voffset=10ex 

\maketitle
\begin{abstract}
In this paper, we derive a Chen--Strichartz formula for stochastic 
differential equations driven by L{\'e}vy processes, that is,
we derive a series expansion 
of the logarithm of the flowmap of the stochastic differential 
equation in terms of commutators of vector fields with stochastic 
coefficients, and we provide an explicit formula for the components in this series. The stochastic components are generated by the L{\'e}vy processes that drive the stochastic differential equation and their quadratic variation and 
 power jumps; the vector fields are given as linear 
combinations of commutators of 
elements in
the pre-Lie Magnus expansion generated by the original vector fields governing our stochastic differential equation.
In particular, we show the 
logarithm of the flowmap is  a Lie series. These results extend previous results
for deterministic differential equations and continuous stochastic differential equations. For these, the Chen--Strichartz series has shown to play a pivotal 
role in the design of numerical integration schemes that 
preserve qualitative properties of the solution such as the
construction of geometric numerical schemes and in the context of efficient numerical schemes.

\keywords{L{\'e}vy processes, stochastic flow, exponential Lie series, Chen--Strichartz series, pre-Lie Magnus series}
\end{abstract}

%%%%%%%%%%%%%%%%%%%%%%%%%%%%%%%%%%%%
%%%%%%%%%%%%%%%%%%%%%%%%%%%%%%%%%%%%

\section{Introduction}
\label{introduction}

We consider ${\mathbb R}^N$-valued homogeneous It\^o stochastic differential systems
driven by L{\'e}vy processes.
 Our aim is to answer the question whether for It{\^o} stochastic differential equations driven by L{\'e}vy processes
 the logarithm of the flowmap can be represented as a stochastic series
in commutators of vector fields, in other words as a Lie series.
 Such a series expression for the flowmap is known as exponential Lie series and its representations in terms of commutators of vector fields as Chen--Strichartz formula.  We further aim to derive an explicit formula for the components of this series. 
For deterministic and continuous stochastic differential equations 
this series has been shown to play a key role in the design of numerical integration schemes that preserve 
qualitative properties of the solution to the given equation. 
The development of the exponential Lie series for deterministic systems originates in the 1950s in the work by Magnus \cite{Magnus} 
and later Strichartz \cite{Strichartz}, and 
 for the stochastic setting for Wiener-driven stochastic differential equations understood in the Fisk--Stratonovich sense by Azencott \cite{Azencott} and Ben Arous \cite{Ben Arous}, and has since been developed and applied in a variety of settings. For the design of numerical integration schemes for Wiener-driven stochastic differential equations, see for example
Castell \& Gaines \cite{CasGai95,CasGai96}, Burrage \& Burrage \cite{BB}, Malham \& Wiese \cite{Lie_Group_paper}, Lord, Malham \& Wiese \cite{LMW2008}, and Armstrong \& King \cite{Armstrong_King}.      

This paper builds on our previous work 
in \cite{Lie_Group_paper},
where for continuous semimartingales 
the algebraic structure
that underlies the flowmap and its computations was established and the exponential Lie series was derived.
Here, we consider stochastic differential equations 
driven by L{\'e}vy processes  of the  form
\begin{equation*}
	Y_t=Y_0+\int_0^t V_0(Y_{s-}){\mathrm d}s
		+ \sum_{i=1}^d\int_0^t V_i(Y_{s-})\,\mathrm{d} W_s^i
		+ \sum_{i=d+1}^{\ell}\int_0^t V_i(Y_{s-})\mathrm{d}J^{i}_s,\qquad t\le T,
\end{equation*}
where $T<\infty$ is the time  horizon, and where the solution process $Y_t$ is $\mathbb R^N$-valued. 
For each $i=1,\ldots,d$, the  $W_t^i$ are independent Wiener processes, and for $i=d+1,\ldots,\ell$ the $J^{i}$ are 
the purely discontinuous martingale parts of the given L{\'e}vy processes. We assume these are independent and
have moments to all orders. We can then express the processes $J^i$ in the form
\begin{equation*}
	J^i_t=\int_0^t\int_{\mathbb{R}}v\,\bar{Q}^i(\mathrm{d}v,\mathrm{d}s),
\end{equation*}
where the
$\bar{Q}^i(\mathrm{d}v,\mathrm{d}s)\coloneqq Q^i(\mathrm{d}v,
\mathrm{d}s)-\rho^i(\mathrm{d}v)\mathrm{d}s$ are random signed measures.
The $Q^i$ are Poisson measures on $\mathbb{R}\times\mathbb{R}_+$ 
with intensity measures $\rho^i(\mathrm{d}v) \mathrm{d}s$
 with $\rho^i(0)=0$ 
and $\int_{\mathbb{R}} (1\wedge v^2) \rho^i(\mathrm{d}v)<\infty$. 
See  for example Applebaum~\cite{Applebaum2004}.
We assume the governing autonomous vector fields 
$V_i\colon\mathbb{R}^N\to\mathbb{R}^N$, $i=0, \ldots, \ell$,  are analytic, and hence a strong solution $Y$ to our {stochastic differential equation} (SDE) exists.
Since the vector fields $V_i$ are analytic, the relevant  information
about the stochastic differential equation is encoded by defining the flowmap as the map 
$\varphi_t$ acting on analytic functions defined as
$(\varphi_t\circ f)(Y_0):=f(Y_t)$.
 Starting point is the 
stochastic Taylor expansion for L{\'e}vy-driven stochastic differential equations. The analyticity of the governing vector fields
allows to obtain a stochastic Taylor expansion in the following separated 
form (Section \ref{WordEncoding}, Eqn.~(\ref{flowmap}))
\begin{equation*}
	\varphi_t = \sum_{w\in \Ab^\ast} I_w(t)D_w.
\end{equation*} 
Each element in this series is the product of an iterated It{\^o} integral $I_w(t)$ with a differential operator $D_w$ that is 
given as a composition of differential operators generated by the vector fields governing the equation. Here the sum is taken over all words over a suitably constructed  alphabet $\Ab$.

This representation allows to encode the flowmap as an element in 
 the complete tensor product of two algebras. One algebra, the \textit{quasi-shuffle algebra}, represents the algebra generated by the iterated It{\^o}
integrals with multiplication (see Section \ref{Qshufflefor}), the other algebra, the \textit{concatenation algebra}, represents 
the algebra generated by the differential operators with composition as product. 
Going one step further, one can use graded duality and properties of the quasi-shuffle Hopf algebra, which is the quasi-shuffle algebra equipped with the coproduct dual to concatenation. 
One thus encodes the flowmap as an element in 
 the complete tensor product of two Hopf algebras in duality.
Such a
separated form and representation was established in \cite{Curry_Levy} and used there to design efficient numerical integrators.

Still one step further, elements in the tensor product of two vector spaces in duality can be understood as linear endomorphisms. This ultimately allows to encode the flowmap by the identity map of the quasi-shuffle Hopf algebra (Prop.~\ref{idrepflow}). This is the point of view we systematically use in the present paper. 
From this point of view, Eqn.~(\ref{flowmap}) below is the expression of the flowmap associated to a \textit{given} linear basis of this Hopf algebra, associated to words over the alphabet $\Ab$. \textit{Any} other choice of a basis will provide \textit{another} expansion of the flowmap associated to a L\'evy-driven stochastic differential equation. 
See also 
\cite{ELMMW2012}, where the (shuffle) convolution algebra 
and the encoding of the flowmap as the identity in this algebra
was established as a natural setting for the design and analysis 
of numerical integration methods for Wiener-driven stochastic differential 
equations.

A novel idea in this paper is to use this approach to define a change of  
basis such that the stochastic processes corresponding to the new basis elements, obtained as linear combinations of products of the original multiple It{\^o} integrals, will follow the classical integration by parts formula rather than It{\^o}'s product rule, that applies to the multiple integrals $I_w$. For continuous processes this change of basis describes 
the transformation from It{\^o} integral to Fisk--Stratonovich integral. 
However, for L{\'e}vy-driven stochastic differential
equations, this transformation does not result in the Fisk--Stratonovich integral, which - in the presence of jumps - does not satisfy the usual integration by parts rule, see for example Theorem V.21 in \cite{Protter}. The new integrals we introduce seem thus interesting on their own, their introduction is another contribution of the paper.

Importantly, in this process we do not change
the Hopf algebra, it is its linear basis that is changed. 
The new basis gives rise to new iterated stochastic integrals, but also to new differential operators. Using Hopf algebraic properties, see \cite{FoissyPatras24} and \cite{PR04}, we show these new differential operators are in fact compositions of vector fields, which we call the
renormalised vector fields. 
 Having performed these steps, we are in familiar territory, allowing us to extend our results in \cite{exp_Lie_series}, Theorem 4.8 and Corollary 4.11, for 
continuous semimartingales to L{\'e}vy-driven stochastic differential equations 
to obtain 
a Chen--Strichartz formula for the logarithm of the flowmap. In 
particular, we show that the exponential series is an 
exponential Lie series.
Furthermore, we derive an explicit representation 
of the renormalised vector fields corresponding to each of the new basis elements. We show that these are given by the components of the pre-Lie Magnus expansion generated by the
original vector fields $V_i$, $i=0,\ldots , \ell$. 

To summarise, we derive the following new results:
\begin{enumerate}
\item we introduce new algebraic ideas and techniques to study flowmaps for stochastic differential equations;
\item we show the exponential series is a Lie serie, also for L{\'e}vy-driven stochastic differential equations;
\item we derive an explicit formula for the Chen--Strichartz series for the logarithm of the flowmap;
\item we obtain a generalisation for L{\'e}vy processes of the It{\^o} to Fisk--Stratonovich transformation (but the new integrals are not Fisk--Stratonovich integrals in the usual sense);
\item we show the vector field components in the 
Chen--Strichartz formula are explicitly given by the components of the
pre-Lie Magnus expansion generated by the original vector
fields governing our stochastic differential equation.  
\end{enumerate}

The article is organised as follows.
In Section \ref{sec:stochTaylor} we derive the separated Taylor expansion for the flowmap. The alphabet and quasi-shuffle algebra are defined in Sections \ref{WordEncoding} and   
\ref{Qshufflefor}. In Section \ref{duality} we provide the algebraic
details for the duality of concatenation Hopf algebra and quasi-shuffle Hopf algebra. The key change-of-basis argument is provided in Section \ref{changeb}, and the new renormalised vector fields are introduced in Section \ref{sec:rvf}, where we derive the explicit expression
of the renormalised vector fields through the pre-Lie Magnus expansion of the vector fields $V_i$, $i=0, \ldots ,\ell$. 
In Section \ref{sec:Chen_Strichartz} we present the Chen--Strichartz formula. We conclude with some remarks in Section \ref{sec:conclusion}.

%%%%%%%%%%%%%%%%%%%%%%%%%%%%%%%%%%%%
%%%%%%%%%%%%%%%%%%%%%%%%%%%%%%%%%%%%

\section{Stochastic Taylor expansion}
\label{sec:stochTaylor}

Key to describing the flowmap is its stochastic Taylor expansion. We briefly recall the derivation and specific properties for 
L{\'e}vy-driven stochastic differential equations following \cite{Curry_Levy}. It\^o's formula states for a real-valued function $f\in C^{1,2}(\Rb^N)$ we have
\allowdisplaybreaks
\begin{align*}
	f(Y_t) \,=\,& f(Y_0) + 
\sum_{j=1}^N \int_0^t V_0^j(Y_{s-}) 
\partial_{y_j}f(Y_{s-})\, {\mathrm d} s
+ \sum_{i=1}^d \sum_{j=1}^N\int_0^t V_i^j(Y_{s-}) 
\partial_{y_j}f(Y_{s-})\, {\mathrm d}W_s^i\\
& 
+ \sum_{i=d+1}^\ell \sum_{j=1}^N\int_0^t V_i^j(Y_{s-}) 
\partial_{y_j}f(Y_{s-})\, \mathrm{d}J_s^i\\
& 
 + \frac{1}{2}\sum_{i=1}^d \sum_{k, l =1}^N \int_0^t V_i^k(Y_{s-})V_i^l(Y_{s-})  
\partial^2_{y_ky_l}f(Y_{s-})\, {\mathrm d}[W^i , W^i]_s\\
& +  \sum_{0\le s\le t}\biggl( f(Y_{s})-f(Y_{s-})-
\sum_{i=d+1}^\ell \bigl(V_i\cdot \partial\bigr) f(Y_{s-})\Delta J_s^i\biggr),
\end{align*}
where $[\cdot, \cdot]$ denotes the \textit{quadratic covariation} or \textit{square bracket} and where $ \Delta J_s\colon=\bigl(\Delta J_s^{d+1}, \ldots, \Delta J_s^\ell\bigr)^\dag$ is the vector of jumps at time $s$.
See Protter \cite{Protter}, Theorem II.31.
Here we have used that the Wiener processes are independent. The last term in this sum is given by
\begin{align*}
\lefteqn{
\sum_{0\le s\le t}\biggl( f(Y_{s})-f(Y_{s-})-
\sum_{i=d+1}^\ell \bigl(V_i\cdot \partial\bigr) f(Y_{s-})\Delta J_s^i\biggr)}\\
 &\, =  
\sum_{0\le s\le t}  \biggl(f\bigl(Y_{s-}+  \sum_{i=d+1}^\ell V_i(Y_{s-})\Delta J^i_s\bigr)-f(Y_{s-})-
\sum_{i=d+1}^\ell \bigl(V_i\cdot \partial\bigr) f(Y_{s-})\Delta J_s^i\biggr)\\
&\, = \sum_{0\le s\le t} \biggl(F(Y_{s-}, \Delta J_s) - F(Y_{s-}, 0)-
 \sum_{i=d+1}^\ell \bigl(V_i\cdot \partial\bigr) f(Y_{s-})\Delta J_s^i\biggr),
\end{align*}
where $F$ denotes the shifted function 
\[
	F(y, v):=  f\bigl(y+\sum_{i=d+1}^\ell v_i V_i(y)\bigr) - f(y)
\]
defined for $(y, v)\in {\mathbb R}^N\times {\mathbb R}^{\ell-d}$ with the convention $v=(v_{d+1},\dots,v_l)$ for indices. Note that we have normalised $F$ so that $F(Y_{s-}, 0)=0$ but have kept the term in the formulas for greater clarity before using Taylor expansions.

The Taylor expansion of $F$ in $v$ is given by
\begin{equation*}
	F(y,v)=\sum_{\gamma} \frac{1}{\gamma !}\partial^{\gamma} F(y, 0)\, v^\gamma,
\end{equation*}
where the summation is over all multi-indices $\gamma=(\gamma_{d+1}, \ldots ,\gamma_\ell)$. Here we define $v^\gamma:=\Pi_{i=d+1}^\ell v_i^{\gamma_i}$ and similarly  $\partial^\gamma:=\Pi_{i=d+1}^\ell \partial_{i}^{\gamma_i}$. Since independent L{\'e}vy processes do not jump at the same time, i.e.~$\sum_{s\le t} \Delta J_s^i\Delta J_s^j 
\equiv 0$ for $i\not=j$, the mixed product terms in the Taylor expansion for $F$ are zero when $F$ is evaluated in $v=\Delta J_s$, and hence we have
\begin{align*}
	F(y,v)\big|_{v=\Delta J_s}
	\, = \, &
	\sum_{i=d+1}^\ell \sum_{m\ge 1} \frac{1}{m!}\partial_i^{m} F(y, 0) v_i^{m}\big|_{v=\Delta J_s}\\
	\,=\, & \sum_{i=d+1}^\ell \sum_{m\ge 1}
	\sum_{j_1 , \ldots , j_m = 1}^N \frac{1}{m!} v_i^m V_{i}^{j_1}(y)\cdots V_{i}^{j_m}(y) \partial^m_{j_1 \cdots j_m}\circ f(y)\big|_{v=\Delta J_s}.
\end{align*}
We define the differential operators $D_{i^{(m)}}$ for $m \geq 1$ as
\begin{align}
\label{eq:diffopera}
	D_{i^{(m)}} 
	& \coloneqq\sum_{j_1 , \ldots , j_m = 1}^N \frac{1}{m!} V_{i}^{j_1}\cdots V_{i}^{j_m} \partial^m_{j_1 \cdots j_m}. 
\end{align}
Equivalently, we have  $D_{i^{(m)}} = \sum_{|\alpha|=m} \frac{1}{\alpha!} V_i^\alpha \partial^\alpha$, where the sum is over all multi-indices $\alpha$. Since $D_{i^{(1)}}= V_i\cdot \partial$, the last term in the It{\^o} formula for $f$ can be expressed more distinctly  as follows
\begin{align*}
\lefteqn{\sum_{0\le s\le t} \biggl(F(Y_{s-}, \Delta J_s) - F(Y_{s-}, 0)-
 \sum_{i=d+1}^\ell \bigl(V_i\cdot \partial\bigr) \circ f(Y_{s-})\Delta J_s^i\biggr) } \\
& = \,   
\sum_{i=d+1}^\ell \sum_{m\geq 2}\sum_{0\le s\le t} 
D_{i^{(m)}}\circ f(Y_{s-}) \bigl(\Delta J_s^i\bigr)^m \\
& = \,  
\sum_{i=d+1}^\ell \sum_{m\geq 2} \int_0^t
D_{i^{(m)}}\circ f(Y_{s-}) {\mathrm d} [J^i]_s^{(m)}.
\end{align*}
Here $[J^i]_s^{(m)}$  is the $m$-th power bracket  of  $J^i$, defined as follows, see \cite{Jamshidian}. 

\begin{definition}[Power bracket] 
For a semimartingale $X$, we define $[X]^{(1)}:=X$, and we define the \textit{power bracket} for $m\ge 2$ by $[X]^{(m)}=[X, [X]^{(m-1)}]$. 
\end{definition}

For $m\ge 3$ the $m$-bracket is given by $[X]^{(m)}= \sum_s \bigl(\Delta X_s\bigr)^m$. This is also known as the \textit{power jump process} (see \cite{Nualart2000}). Since the $J^i$ are purely discontinuous martingales, we have for $m\ge 2$ 
\[
	[J^i]^{(m)}_t=
	\sum_{0\le s \le t} \bigl(\Delta J^i_s\bigr)^m = \int_0^t \int_{\mathbb R} v^m \,Q^i({\mathrm d}v, {\mathrm d}s).
\]
In conclusion, It\^o's Formula yields for analytic functions $f$
\begin{align*}
	f(Y_t) \,=\,& f(Y_0) + 
	\sum_{j=1}^N \int_0^t V_0^j(Y_{s-}) 
	\partial_{y_j}f(Y_{s-})\, \mathrm{d} s
	+ \sum_{i=1}^d \sum_{j=1}^N\int_0^t V_i^j(Y_{s-}) 
	\partial_{y_j}f(Y_{s-})\, \mathrm{d}W_s^i\\
	& 
	+ \sum_{i=d+1}^\ell \sum_{j=1}^N\int_0^t V_i^j(Y_{s-}) 
	\partial_{y_j}f(Y_{s-})\, \mathrm{d}J_s^i\\
	& 
 	+ \frac{1}{2}\sum_{i=1}^d \sum_{k, l =1}^N \int_0^t V_i^k(Y_{s-})V_i^l(Y_{s-})  
	\partial^2_{y_ky_l}f(Y_{s-})\, {\mathrm d}[W^i , W^i]_s\\
	& + \sum_{i=d+1}^\ell  \sum_{m\geq2} \int_0^t
	D_{i^{(m)}}\circ f(Y_{s-}) {\mathrm d}[J^i]^{(m)}_s.\\
\end{align*}

%%%%%%%%%%%%%%%%%%%%%%%%%%%%%%%%%%%%
%%%%%%%%%%%%%%%%%%%%%%%%%%%%%%%%%%%%

\section{Word encoding}
\label{WordEncoding}

This section and the next one provide the algebraic background for what is to come. Since the vector fields governing our 
stochastic equation are analytic, we define 
the flowmap $\varphi_t$ as the random map prescribing the transport of the initial condition $f(Y_0)$ to the solution $f(Y_t)$ at time $t>0$ for analytic functions $f:{\mathbb R}^N \to {\mathbb R}$.

\smallskip

It\^o's Formula shows that the constituting factors of the flowmap are given by the driving processes (time $t$, the Wiener processes and the purely discontinuous martingales), and the power brackets of the latter, as well as the governing vector fields $V_i$ and the differential operators $D_{i^{(m)}}$, for $m\ge 2$. It is known, see \cite{quasi_shuffle}, that the algebra of repeated integrals of semimartingales is a quasi-shuffle algebra. In \cite{Curry_Levy} it was shown that the flowmap can be represented in a Hopf algebra framework. Proceeding similarly, we define the countable alphabet ${\mathbb A}$ as follows:
\begin{equation*}
	{\mathbb A}:=\{ 0, 1, \, \ldots \, ,\, \ell\}
	\cup \{ i^{(2)} \, :\, i=1, \ldots , d\} 
	\cup \{ i^{(m)}:\, i=d+1, \ldots, \ell,\, m\ge 2\}.
\end{equation*}
For notational convenience we will also write $i^{(1)}=i$, for $i=0, \ldots, \ell$. 
We set 
\[
\begin{array}{lll}
	I_0 & :=t & \\
	I_i &  := W_i & \mbox{ for } i=1, \ldots , d\\[0.1cm] 
	I_i & := J_i & \mbox{ for } i=d+1, \ldots, \ell\\[0.1cm]
	I_{i^{(2)}} & :=  [W^i, W^i] & \mbox{ for } i=1, \ldots, d\\[0.1cm]
	I_{i^{(m)}} & :=[J^i]^{(m)} & \mbox{ for } i=d+1, \ldots, \ell \mbox{ and } m\ge 2,
\end{array}
\]
and we set
\[
\begin{array}{lll}
	D_i & := V_i\cdot \partial & \mbox{ for } i=0, \ldots ,\ell\\[0.1cm]
	D_{i^{(2)}} & := \frac{1}{2} \sum_{j, k=1}^N V_i^jV_i^k \partial^2_{y_jy_k} &\mbox{ for } i=1, \ldots, d.
\end{array}
\]
For $i= d+1, \ldots, \ell$ and $m\ge 2$
we recall the definition of the differential operators:
\[
	D_{i^{(m)}} = \sum_{|\alpha|=m} \frac{1}{\alpha!} V_i^\alpha \partial^\alpha.
\]
We have thus established a one-to-one correspondence between   
the multiple It{\^o} integrals and the alphabet $\Ab$ as well as
a one-to-one correspondence between the differential operators 
and the alphabet $\Ab$.

\begin{remark}
\begin{enumerate}
\item We emphasise that the operators $D_a$, $a\in {\mathbb A}$, are differential operators.  
\item The alphabet ${\mathbb A}$ is countably infinite.
\item While $[W^i, W^i]_t=t$, for $1\le i \le d$, the encoding in the alphabet ensures the information from which Wiener process the bracket $[W^i, W^i]$ is generated will be retained. 
\end{enumerate}
\end{remark}

With this encoding, It\^o's Formula simplifies to 
\[
	f(Y_t)= f(Y_0)+ \sum_{i\in {\mathbb A}} \int_0^t D_i \circ f(Y_{s-}) {\mathrm d} I_i(s).
\]
We denote by $\Ab^\ast$ the free monoid of words on the alphabet $\Ab$. We denote by $\mathbf 1$ its unit, the empty word, to distinguish it from the letter $1\in \mathbb A$. Using iteratively It\^o's Formula yields the stochastic Taylor expansion (see, for instance, \cite{Kloeden_Platen} and 
\cite{Bruti_Liberati_Platen})
\begin{equation*}
\label{Taylor_expansion}
	f(Y_t) = \sum_{w\in\Ab^{\ast}} I_w(t)D_w\circ f(Y_{0}) .
\end{equation*}
Here, for a word $w=a_1\cdots a_n$ the term $I_w(t)$ denotes the multiple It\^o integral 
\begin{equation*}
	I_w(t)=\int_{0\leqslant\tau_{1}\leqslant\cdots\leqslant\tau_{n}\leqslant t} 
	\,\mathrm{d} I_{a_1}(\tau_1)\cdots\,\mathrm{d} I_{a_n}({\tau_n}),
\end{equation*}
and  $D_w$ denotes the composition of the differential operators $D_{a_i}$, $0\le i \le n$,
\begin{equation*}
	D_w= D_{a_1}\circ\cdots\circ D_{a_n}.
\end{equation*}
Hence the flowmap $\varphi_t$ can be expressed as
\begin{equation}
\label{flowmap}
	\varphi_t = \sum_{w\in \Ab^\ast} I_w(t)D_w.
\end{equation}

This representation of the flowmap has the advantage that the stochastic information given by the multiple integrals 
$I_w$ and the geometric information given by the differential operators $D_w$ are separated. See \cite{Curry_Levy}, where this representation was derived and utilised to design efficient integrators for L{\'e}vy stochastic differential
equations.

%%%%%%%%%%%%%%%%%%%%%%%%%%%%%%%%%%%%
%%%%%%%%%%%%%%%%%%%%%%%%%%%%%%%%%%%%

\section{Quasi-shuffle formalism}
\label{Qshufflefor}

The product of It{\^o} integrals can be encoded as a quasi-shuffle product on an appropriately chosen alphabet, while the composition of vector fields can be encoded as a concatenation product. This, together with the word encoding of the flowmap, allows to further abstract symbolically the latter as a complete tensor product in the tensor product of two Hopf algebras \cite{quasi_shuffle}.

Following a very similar approach, we choose here to encode the flowmap by a linear endomorphism of a Hopf algebra.
See also \cite{Curry_Levy}, where the convolution (shuffle) algebra of linear endomorphisms was shown to be the natural setting for the design of numerical integration schemes.
The object encoding the flowmap is the identity map of a Hopf algebra of words. However, as we shall see, this simple encoding is extremely powerful. Indeed, it allows in particular to deduce subtle properties of the flowmap for L\'evy-driven stochastic differential equations from relatively standard algebraic constructions and basis change arguments. The encoding is described in this section and the forthcoming one. Later sections will be devoted to applications.

Let $\Rb\la\Ab\ra$ denote the non-commutative polynomial algebra over $\Ab$ generated by monomials (or words) that can be constructed from the alphabet $\Ab$. Equivalently, $\Rb\la\Ab\ra$ is the linear span of $\Ab^\ast$, with $\mathbf 1$ as its algebra unit. We encode the quadratic covariation processes of the L{\'e}vy  processes driving the 
stochastic differential equation as a product  of letters in $\Ab$, which we will also denote by $[\cdot, \cdot]$. 

\begin{definition}[Bracket product on $\Ab$]
\label{def:bracketprod}
We define a commutative associative product $[\cdot , \cdot ]$  on the vector space over ${\mathbb R}$ generated by the alphabet $\Ab$ as follows:
\begin{enumerate}
\item if $i=0$ then $[i, j]:=\mathbf{0}$ for all $j\in \Ab$, 
\item 
$[i, i]:=i^{(2)}$ and $[i, i^{(2)}]:=\mathbf{0}$ for all $i=1, \ldots , d$, 
\item $[i^{(l)}, i^{(k)} ]:=i^{(l+k)}$ for all $l, k\ge 1$ and for all  $i=d+1, \ldots ,\ell$, and
\item $[i^{(l)}, j^{(k)} ]:=\mathbf{0}$ for $i\not=j$ and all $l, k\ge 1$.
\end{enumerate}
We extend this definition to the vector space generated by
 $\Ab$ linearily. Here $\mathbf{0}$ denotes
the zero in $\mathbb R$ to distinguish it from the letter $0$ in $\Ab$.
For  letters $a_i$ in $\Ab$, iterated products are denoted by 
$$
	[a_1,\dots,a_n]:=[a_1,[a_2,\dots,a_n]].
$$
\end{definition}

We define a grading $g$ on $\Rb\la\Ab\ra$ as follows, see also \cite{quasi_shuffle}.

\begin{definition}[Grading]
We assign the letter $0$ the grading $g(0)=1$, and for  $i=1, \ldots, \ell$, and $k \ge 1$ we set $g\bigl(i^{(k)}\bigr)  =  k$. For a word $w\in \Rb\la\Ab\ra$ we set $g(w)$ to be the sum of the gradings of each of its letters.  
\end{definition}

\begin{remark}\label{rm:bracket} 
\begin{enumerate}
\item
The bracket product on $\Ab$ encodes the quadratic covariation of the L{\'e}vy processes governing our stochastic differential equation. Its associativity, which is easily checked, reflects associativity properties of quadratic covariation calculus. The fourth property in Definition \ref{def:bracketprod} reflects the property of independent L\'evy processes that with probability $1$ jumps do not occur at the same time. In other words, we have
\[
	\sum_{s\le t}\Delta J^i_s\,\Delta J^j_s\equiv 0 \mbox{ for } i\not=j.
\]
\item
The properties in Definition \ref{def:bracketprod} imply that the set
\[
	\bigl\{j, k: [i^{(j)}, i^{(k)}]=i^{(m)}\bigr\}=\bigl\{j, k: j+k=m\bigr\}
\]
is finite for all $m\ge 1$. More generally, if $w=i^{(j_1)}\cdots i^{(j_n)}$ is a word such that $[i^{(j_1)}, \ldots , i^{(j_n)}]=i^{(m)}$, then $j_1 + \cdots + j_n=m$. 
\item We note that for each grade $k\ge 1$ there are finitely many letters of grade $k$. Hence $\Ab$ equipped with the grading $g$ is locally finite (that is, each graded component is a finite set). Furthermore, the grading $g$ is additive on the bracket of two letters: for  $[i^{(j)}, i^{(k)}]\not={\mathbf 0}$, we have $g\bigl([i^{(j)}, i^{(k)}]\bigr)=j+k=g\bigl(i^{(j)}\bigr) +g\bigl( i^{(k)}\bigr)$. 
\item Alternatively, and without affecting the results, we could have defined the bracket product on $\Ab \,\cup\, {\mathbf 0}$, resulting in a semigroup structure and thus providing a slightly different angle on the construction. Note that the semigroup $\Ab \,\cup\, {\mathbf 0}$ is not graded due for example to the first relation in Definition \ref{def:bracketprod} and the fact that $\mathbf 0$ is necessarily of degree 0. However, when the product of two letters is nonzero, it is compatible with the grading and this allows to automatically apply to $\Rb\la\Ab\ra$ the constructions in \cite{Hoffman2000}.
\end{enumerate}
\end{remark}

\begin{definition}[Quasi-shuffle product]
The commutative quasi-shuffle product $\star$ on $\Rb\la\Ab\ra$  is  inductively defined by
\begin{align}
\label{qshprod}
\begin{aligned}
	u\,& \star\, {\mathbf 1}
	= {\mathbf 1}\,\star\, u
	=u, \\
	ua \star vb & 
	=(u \star vb)a + (ua \star v)b + (u \star v)[a,b],
\end{aligned}
\end{align}
where $u$ and $v$ are words and $a$ and $b$ are letters, and where
 `$ {\mathbf 1}$' is the empty word.
\end{definition}

By \cite{Hoffman2000,NR}, the quasi-shuffle product \eqref{qshprod} equips $\Rb\la\Ab\ra$ with an associative and commutative graded algebra structure.

The following proposition accounts for the role of quasi-shuffles in stochastic calculus. See, e.g.,~\cite{Curry_Levy} for further  references on the rules of stochastic calculus applying to L\'evy 
stochastic differential equations.

\begin{proposition}[Word-to-integral map]\label{def:wordmaps}
For words $w\in \Ab^\ast$, we define
 the word-to-integral map $\mu$ by
\[
	\mu (w):=I_w
\]
and extend this definition to $\Rb\la\Ab\ra$ linearily.  The map $\mu$ is an algebra morphism from the quasi-shuffle algebra $\bigl(\Rb\la\Ab\ra, \star)$ to the algebra generated by the multiple integrals $I_w$, that is, for words $w,w'$ in $\Ab^\ast$ we have
$$
	\mu(w \star w')=I_wI_{w'}.
$$
\end{proposition}

%%%%%%%%%%%%%%%%%%%%%%%%%%%%%%%%%%%%
%%%%%%%%%%%%%%%%%%%%%%%%%%%%%%%%%%%%

\section{Hopf algebra duality}
\label{duality}

We remark that the process of integration is dual to differentiation. We encode this general fact by introducing a dual alphabet $\overline\Ab$. Letters in and words over the alphabet $\overline\Ab$ are denoted $\overline a$, and 
$\overline\omega=\overline a_1\cdots \overline a_n$, respectively, where $a$ is a letter in $\Ab$ and $\omega=a_1\cdots a_n$ is a word in $\Ab^\ast$. We make an exception for the empty word that we denote with the same symbol $\mathbf 1$ in $\Ab^\ast$ and in $\overline\Ab^\ast$.

Recall that given two words $a_1\cdots a_n$ and $b_1\cdots b_m$, their concatenation is the word  $a_1\cdots a_n b_1\cdots b_m$. The vector space $\Rb\la\overline\Ab\ra$ equipped with the concatenation product of words, that we denote $\cdot$ or $\mathrm{conc}$, is the free associative algebra generated by $\Ab$.

\begin{definition}[Word-to-differential operator map]
\label{def:wordmaps2}
For words $\overline w\in \overline\Ab^\ast$, we define the word-to-differential map $\bar{\mu}$ by
\[
	\bar{\mu} (\overline w):=D_w.
\]
We extend this definition to $\Rb\la\overline\Ab\ra$ linearly. 
\end{definition}

As the product of differential operators is non-commutative and associative, and as $\Rb\la\overline\Ab\ra$ is free, we have the following result.

\begin{proposition}
The map $\bar{\mu}$  is an algebra morphism from the algebra $\bigl(\Rb\la\overline\Ab\ra, \mathrm{conc})$
to the composition algebra generated by the differential operators $D_w$.
\end{proposition}

Recall that the graded dual of a graded vector space $V=\bigoplus_{n\in\Nb}V_n$ is the graded vector space $V^\ast=\bigoplus_{n\in\Nb}V^\ast_n$. When the $V_i$ are finite dimensional and equipped with a basis $B_i$, the union of the dual bases $B_i^\ast$ of the $V_i^\ast$ forms a basis of $V^\ast$.

The sets of words $\Ab^\ast$ and $\overline\Ab^\ast$ form bases of 
$\Rb\la\Ab\ra$ and $\Rb\la\overline\Ab\ra$, respectively. 
Setting $\la \omega,\overline\omega'\ra:=\delta_{\omega=\omega'}$ puts the two bases and the corresponding vector spaces in graded duality. Note that they are not in duality, as they are infinite dimensional. Hence, by ``dual" we will therefore always refer to graded duality from now on when dealing with graded infinite dimensional spaces.

One can then take advantage of duality to define the deconcatenation coproduct, denoted $\Delta$, on $\Rb\la\Ab\ra$ as the dual of the concatenation product on $\Rb\la\overline\Ab\ra$:
\begin{align}
\label{eq:adjqshcoprod1}
	\Delta(w) 
	&= \sum_{w_1,w_2 \in \Ab^*} \langle w, \overline w_1\cdot \overline w_2 \rangle w_1 \otimes w_2.
\end{align}
Equivalently, for the word $w=a_1\cdots a_n$, we have
$$
	\Delta(w)=\sum\limits_{i=0}^n a_1\cdots a_i\otimes a_{i+1}\cdots a_n,
$$
with the convention $a_1\cdots a_0:=\mathbf 1$ and $a_{n+1}\cdots a_n:=\mathbf 1$.

The following classical result seems to have appeared first in \cite{NR}. See also \cite{Hoffman2000}. We refer to \cite{CP21} for general definitions and properties of Hopf algebras that we implicitly or explicitly use below.

\begin{proposition}
\label{prop:quasishuffleHoopf}
The graded vector space $\Rb\la\Ab\ra$ equipped with the quasi-shuffle product \eqref{qshprod} and deconcatenation coproduct \eqref{eq:adjqshcoprod1} forms a graded Hopf algebra, denoted  $\bigl(\Rb \la\Ab\ra, \star,\Delta\bigr)$.
\end{proposition}

One can also define the de-quasi-shuffle coproduct   denoted $\delta$ on $\Rb\la\overline\Ab\ra$ as the dual of the quasi-shuffle product on $\Rb\la\Ab\ra$. For a letter $a \in \Ab$ we have
\begin{equation}
\label{eq:adjqshcoprod2}
	\delta(\overline a) 
	= \overline a\otimes \mathbf 1+\mathbf 1\otimes \overline a
	+ \sum_{a_1,a_2 \in \Ab} \langle a, \overline{ [a_1,a_2]} \rangle \overline a_1 \otimes \overline a_2.
\end{equation}
For a word $\overline w=\overline a_1\cdots \overline a_n$, the coproduct is obtained as
$$
	\delta(\overline w):=\delta(\overline a_1)\cdot \cdots\cdot \delta(\overline a_n),
$$
where $\cdot$ now denotes the product in $\Rb\la\overline\Ab\ra\otimes \Rb\la\overline\Ab\ra$ (the tensor product of the concatenation products on each copy of $\Rb\la\overline\Ab\ra$).

By local finiteness of the graded monoid $\overline\Ab^\ast$ or by observing directly that each letter de-quasi-shuffles in a finite number of ways, the de-quasi-shuffling operator $\delta$ is well-defined.

\begin{remark}
Going back to Proposition \ref{prop:quasishuffleHoopf}, it follows from its definition that the quasi-shuffle Hopf algebra $\bigl(\Rb \la\Ab\ra, \star,\Delta\bigr)$ has as its (graded) dual the concatenation Hopf algebra $\bigl(\Rb \la\overline\Ab\ra, \mathrm{conc}, \delta\bigr)$ with the de-quasi-shuffle coproduct \eqref{eq:adjqshcoprod2}. 
\end{remark}

Notice now that the identity map denoted $\mathrm{Id}$ of $\Rb \la\Ab\ra$ can be represented in the basis of words $\Ab^\ast$ as
$$
	\mathrm{Id}=\sum\limits_{w\in \Ab^\ast}w\otimes \overline w.
$$
One can think of this expression as an infinite diagonal matrix whose entries are indexed by pairs of words in $\Ab^\ast$.
Here we use the convention that the action of a tensor $w\otimes \overline w'$ on $\Rb \la\Ab\ra$ is given by
$$
	(w\otimes \overline w')(w''):=\la w'',\overline w'\ra w.
$$

We can now represent the flowmap $\varphi_t$ as follows. See also \cite{Curry_Levy}  and \cite{exp_Lie_series} for similar expansions for L{\'e}vy stochastic differential equations and for the case of stochastic differential equations driven by continuous semimartingales, respectively.

\begin{proposition}
The flowmap $\varphi$ has the representation
\begin{equation}
	\varphi
	=\mu\otimes\bar{\mu}(\mathrm{Id})
	=\mu\otimes\bar{\mu} \biggl(\sum\limits_{w\in \Ab^\ast} w\otimes \overline w\biggr)
\end{equation}
with multiple It{\^o} integrals on the left ($\Ab^\ast$ terms) and composition of differential operators on the right ($\overline\Ab^\ast$ terms). 
\end{proposition}

%%%%%%%%%%%%%%%%%%%%%%%%%%%%%%%%%%%%
%%%%%%%%%%%%%%%%%%%%%%%%%%%%%%%%%%%%

\section{Change of basis}
\label{changeb}

The key to all our forthcoming arguments is a change of basis argument. In our opinion its interest for stochastic integration theory goes much beyond the present article as it also clarifies various standard constructions such as the It\^o-to-Fisk--Stratonovich transformation for integrals. Recall first the following result from \cite{PR04} which was further developed recently in the work \cite{FoissyPatras24}: 

\begin{proposition}{\cite[Lemma 22]{PR04}}
Let $H$ be a graded connected cocommutative Hopf algebra over $\Rb$ whose graded components are finite dimensional. If $H$ is free as a graded algebra, then the Lie algebra of primitive elements in $H$ is a free Lie algebra.
Furthermore, if $H$ if freely generated by the family $(f_i)_{i \in I}$, one can choose the family $(\log^\ast(\mathrm{Id})(f_i))_{i\in I}$ as a family of primitive free generators.
\end{proposition}

The convolution logarithm of the identity of the Hopf algebra $H$ is the generalisation to arbitrary graded connected commutative or cocommutative Hopf algebras, introduced in \cite{patras1992}, of the classical Eulerian (or Solomon, or canonical) idempotent of the theory of free Lie algebras. In the cocommutative case, $H$ is always an enveloping algebra of a Lie algebra and the map identifies with the projection onto the Lie algebra induced by the Poincaré--Birkhoff--Witt theorem. We refer to \cite{CP21,patras1993decomposition,patras1994algebre} for details and will only use here the resulting formulas in the particular case where $H=(\Rb \la\overline\Ab\ra,\mathrm{conc},\delta).$

\begin{remark}
In \cite{FoissyPatras24}, it was shown that any family of Lie idempotents can actually be used to perform the construction. It would be interesting to understand the meaning of this extension for stochastic calculus in general, and for L\'evy flowmaps in particular.
\end{remark}

\begin{corollary} 
The following elements form a family of algebraic generators of the algebra $(\Rb \la\overline\Ab\ra,\mathrm{conc})$
$$
	\overline a^0 := \sum_{n\ge 1}\frac{(-1)^{n-1}}{n}\sum_{[a_1, \ldots, a_n]=a} \overline a_1 \cdots \overline a_n,
$$
for $a\in \Ab$. Thus, denoting $\overline\Ab^0$ the set of such elements, the two algebras $(\Rb \la\overline\Ab\ra,\mathrm{conc})$ and $(\Rb \la\overline\Ab^0\ra,\mathrm{conc})$ are equal. 
Furthermore, the elements of $\overline\Ab^0$ are primitive for the coproduct $\delta$. Thus, on  $(\Rb \la\overline\Ab^0\ra,\mathrm{conc})$ the expression of the coproduct is given by the unshuffling of words:
$$
	\delta(\overline a_1^0\cdots \overline a_n^0)
	=\sum\limits_{{\{i_1,\ldots,i_k\}\cup \{j_1,\ldots,j_{n-k}\}=\{1,\ldots,n\}}}
	\overline a_{i_1}^0 \cdots \overline a_{i_k}^0 \otimes \overline a_{j_1}^0\cdots \overline a_{j_{n-k}}^0,
$$
for $\overline a_i^0$ in $\overline\Ab^0$.
\end{corollary}

The square bracket process is equal to zero for independent Wiener processes and more generally for independent L\'evy processes as these almost surely do not jump at the same time. This translates algebraically into the fact that the expression for the $(\overline i^{(m)})^0$ involves only letters $\overline i^{(k)}$ with $k\leq m$. More precisely we have the following result.

\begin{lemma}\label{explexpr}
For $m\geq 1$, we have the explicit expression 
\begin{align*}
	(\overline i^{(m)})^0 & =\, 
	\sum_{n=1}^m\sum_{j_1+\cdots + j_n=m}
	\frac{(-1)^{n-1}}{n}\overline i^{(j_1)}\cdots \overline i^{(j_n)}.
\end{align*}
\end{lemma}
We omit the proof as it immediately follows from the definition of the associative commutative bracket product on $\Ab^\ast$.

For example, up to order three we find
\begin{align*}
	(\overline i^{(1)})^0 &= \overline i^{(1)}\\
	(\overline i^{(2)})^0 &= \overline i^{(2)} - \frac{1}{2} \overline i^{(1)}\overline  i^{(1)} \\
	(\overline i^{(3)})^0 &= \overline i^{(3)} - \frac{1}{2} \biggl(\overline i^{(2)} \overline i^{(1)} 
	+ \overline  i^{(1)} \overline i^{(2)}\biggr) + \frac{1}{3} \overline i^{(1)}\overline i^{(1)} \overline i^{(1)} .
\end{align*}

The relation between the $\overline a^0\in \overline\Ab^0$ and the $\overline a\in \overline\Ab$ can be inverted by taking advantage of the properties of the exponential map. Indeed, we have:

\begin{lemma} \label{lemma:exp_dual}
The elements of $\overline\Ab$ can be expanded as
$$ 
	\overline a
	= \sum_{n\ge 1}\sum_{[a_1, \ldots, a_n]=a} 
		\frac{1}{n!} \overline a_1^0 \cdots \overline a_n^0.
$$
\end{lemma}

\begin{definition}
We write $(\Ab^0)^\ast$ for the basis of $\Rb \la\Ab\ra$ dual to the basis $(\overline\Ab^0)^\ast$ of $\Rb \la\overline\Ab\ra$. 
\end{definition}

\begin{remark}
Notice that the expansion of the letters $\overline a^0$ in terms of the $\overline a$, and conversely, immediately provide closed expressions for the basis change from $(\Ab^0)^\ast$ to 
$(\overline\Ab)^\ast$ and conversely.

Although the basis change approach we are following seems the simplest way to deduce them, these formulas could also be deduced from the explicit formula for the Hoffman isomorphism between the quasi-shuffle Hopf algebra and the shuffle algebra over a graded locally finite alphabet. See Theorems 3.3 and 4.1 in \cite{Hoffman2000}. 

The dual bases changes (see below) are indeed precisely given by the formulas for the Hoffman isomorphism and equivalent to it --- up to the fact that the Hoffman isomorphism is usually not stated in terms of bases changes inside a quasi-shuffle Hopf algebra but as an isomorphism between two different Hopf algebras (the shuffle and quasi-shuffle Hopf algebras). We do not detail the computation which follows by duality. Instead, we refer to the recent articles \cite{BFT,FoissyPatras24} for further complementary insights on the Hoffman isomorphism. 
\end{remark}

\begin{lemma}\label{lemma:exp}
The basis change formula expressing the elements of $(\Ab^0)^\ast$ in terms of the elements of $\Ab^\ast$ in $\Rb \la\Ab\ra$ is obtained as
$$
	w^0=\sum_{(i_1, \ldots ,i_\ell)}\frac{1}{i_1!\cdots i_\ell!}(i_1, \ldots,\, i_\ell)\circ w.
$$
where $w^0=a_1^0\cdots a_n^0$. Here, the sum runs over all multi-indices $(i_1, \ldots , i_\ell)$ such that $i_1+ \cdots + i_\ell=|w|$ where $|w|$ denotes the length of the word $w$ and  
\begin{equation*}
	(i_1, \ldots,\, i_\ell)\circ w := 
		[ a_1,\ldots , a_{i_1}]
		[ a_{i_1+1}, \ldots , a_{i_1+i_2}]\cdots
		[ a_{i_1+\cdots+i_{\ell-1}+1}, \ldots , a_{n}]
\end{equation*}
is the concatenation of the repeated bracket product of the letters in each subword as indicated.
\end{lemma}

\begin{lemma}
The basis change formula expressing the elements of $\Ab^\ast$ in terms of the elements of $(\Ab^0)^\ast$ in $\Rb \la\Ab\ra$ is obtained as
$$
	w=\sum_{(i_1, \ldots, i_\ell)}
	\frac{(-1)^{|w|-\ell}}{i_1\cdots i_\ell} (i_1, \ldots,\, i_\ell)\circ w^0.
$$
where $w^0=a_1^0\cdots a_n^0$. Here we use the same notation as above by setting
\begin{equation*}
	(i_1, \ldots,\, i_\ell)\circ w^0 := 
		[ a_1,\ldots , a_{i_1}]^0
		[ a_{i_1+1}, \ldots , a_{i_1+i_2}]^0\cdots
		[ a_{i_1+\cdots+i_{\ell-1}+1}, \ldots , a_{n}]^0
\end{equation*}
in terms of the concatenation of repeated bracket products of the letters in each subword as indicated.
\end{lemma}

\begin{lemma}
With the notation of the previous section, we obtain
$$
	\mathrm{Id}=\sum\limits_{w\in\Ab^\ast}w\otimes\overline w=\sum\limits_{w^0\in(\Ab^0)^\ast}w^0\otimes\overline w^0.
$$
Further, the expression of the coproduct in the basis $(\overline\Ab^0)^\ast$ of $\Rb \la\overline\Ab\ra$ dualises to the expression of the product $\star$  as a shuffle product in the basis $(\Ab^0)^\ast$ of $\Rb \la\Ab\ra$:
\begin{align}
	(a_1^0\cdots a_n^0) \star (a_{n+1}^0\cdots a_{n+m}^0)
	&=\sum\limits_{\sigma\in Sh(n,m)}a^0_{\sigma(1)}\cdots a^0_{\sigma(n+m)}\label{eq:shuff0} \\
	&=:(a_1^0\cdots a_n^0)\sh (a_{n+1}^0 \cdots a_{n+m}^0),\label{eq:shuff}
\end{align}
for  $a_i^0\in\Ab^0$, and where $Sh(n,m)$ stands for the set of permutations of $\{1,\ldots,n+m\}=:[n+m]$ such that $\sigma^{-1}(1) < \cdots < \sigma^{-1}(n)$ and $\sigma^{-1}(n+1)< \cdots < \sigma^{-1}(n+m)$.
\end{lemma}

Hence, we have the following expression for the flowmap.

\begin{proposition}\label{idrepflow}
The flowmap $\varphi$ has the representation
\begin{equation}
	\varphi=\mu\otimes\bar{\mu}(\mathrm{Id})=\mu\otimes\bar{\mu} \biggl(\sum_{w^0} w^0\otimes \overline w^0\biggr),
\end{equation}
where the sum is taken over all words in $(\Ab^0)^\ast$. 
\end{proposition}

Note, however, that the terms in the above expansion, $\mu(w^0)$ and $\bar{\mu}(\overline{w}^0)$, cannot be immediately interpreted in the same way as when using the representation associated with $\Ab$. We will explore the interpretation of these new terms in the following sections.

For consistency with our previous article on the exponential series for stochastic differential equations driven by continuous semimartingales \cite{exp_Lie_series}, we will set:

\begin{definition}[Renormalised flowmap]
\begin{equation}
\label{eq:Jintegrals}
	J_w:=\mu(w^0) 
	\quad\ \text{and} \quad\
	V_w:=\bar\mu(\overline w^0),
\end{equation}
so that the flowmap expansion rewrites as
\begin{equation}
	\varphi_t=\sum_{w} J_w(t) V_w.
\end{equation}
We call this expansion the renormalised expansion of the flowmap.
\end{definition}

To avoid confusion, as the symbol $V_i$ is now used for both a vector field (defined as a function from $\Rb^N$ to $\Rb^N$ introduced at the beginning of the article) and the corresponding differential operator $V_i \cdot \partial$, which equals $\bar{\mu}(\overline{i}^0)$, we will henceforth denote the vector field by $V_i$ and the differential operator $\bar{\mu}(\overline{i}^0)$, as defined in (\ref{eq:Jintegrals}), by $V_{i^{(1)}}$.

The terminology ``renormalised expansion" originates from the fact that for continuous semimartingales the process we are presently investigating for L\'evy processes corresponds to the It\^o-to-Fisk--Stratonovich transformation. The latter is known to be a Wick-type renormalisation process. It can be regarded as a generalisation to continuous semimartingales of the Wick product construction for quantum fields and random variables. For a Hopf algebraic approach to Wick products see e.g.~\cite{EPTZ}.

Let us detail further this observation focussing on the iterated stochastic integral terms $J_w$ . For Wiener processes, or  continuous semimartingales, these are Fisk--Stratonovich integrals \cite{exp_Lie_series}. While we do not have the exact same interpretation for $J_w$ defined in \eqref{eq:Jintegrals}, equations \eqref{eq:shuff0} and \eqref{eq:shuff} imply 
the stochastic iterated integrals $J_w$ obey the usual  integration by parts formula. That is,
in analogy to the continuous case, since the map $\mu$  is an algebra morphism from $\bigl(\Rb\la \Ab\ra, \star)$ onto the algebra generated by the iterated integrals $I_w$, or, equivalently by the $J_w$, we have 
$$
	J_v \cdot J_w=\mu(v^0)\cdot \mu(w^0)  =\mu(v^0\star w^0) =\mu(v^0\sh w^0)= J_{v \sh w},
$$
where in the last term we use the convention that $J_x$, for $x=\lambda_1w_1+\cdots + \lambda_n w_n$ with the $w_i$ being words in $\Ab^\ast$, stands for $\lambda_1J_{w_1}+\cdots + \lambda_n J_{w_n}.$ This result will be essential later on, when we will obtain a Chen--Strichartz formula for L\'evy-driven stochastic differential equations.

%%%%%%%%%%%%%%%%%%%%%%%%%%%%%%%%%%%%
%%%%%%%%%%%%%%%%%%%%%%%%%%%%%%%%%%%%

\section{The renormalised vector fields}
\label{sec:rvf}

Recall that the (pre-Lie) Magnus series classically expresses the derivative of the logarithm of the fundamental solution of a time-dependent linear matrix differential equation. See, for example, \cite{EFMan2009}, \cite[Remark 6.5.3]{CP21} or the survey \cite{EPNacer}. The present section is dedicated to showing that the renormalised differential operator $V_{i^{(m)}}$ is obtained from the pre-Lie Magnus expansion of $V_i$.

\begin{theorem}
\label{thm:vectorfield}
For any letter $i^{(m)}\in \Ab$, $m\ge 1$, the renormalised differential operator 
$$
	V_{i^{(m)}}:=\bar{\mu}((\overline i^{(m)})^0)
$$ 
is a vector field. More precisely, $V_{i^{(m)}}$ is given by the $m$-th component of the pre-Lie Magnus expansion of the vector field $V_{i^{(1)}}=V_i \cdot \partial$.
\end{theorem}

In order to prove the statement let us recall first some general properties of vector fields and pre-Lie algebras. For further details, proofs and references, we refer the reader to the same references as above.

 For two vector fields on $\mathbb{R}^{{N}}$, 
$$
	V(x)=\sum_{j=1}^{{N}} V^j(x)\frac{\partial}{\partial x_{{j}}} 
	\quad\ \text{and} \quad\
	W(x)=\sum_{i=1}^{{N}} W^i(x)\frac{\partial}{\partial x_{{i}}},
$$ 
we define the product
\begin{equation}
\label{preLieVect1}
	(V \triangleright W) (x)=\sum_{i=1}^{{N}}\Big(\sum_{j=1}^{{N}}
	V^j(x)\frac{\partial}{\partial x{_{{j}}}}W^i(x)\Big)\frac{\partial}{\partial x{_{{i}}}}.
\end{equation}
One can easily show that it satisfies the left pre-Lie relation 
\begin{equation}
\label{plid}
	(V \triangleright W) \triangleright U - V \triangleright (W \triangleright U) 
		= (W \triangleright V) \triangleright U - W \triangleright (V \triangleright U).
\end{equation}

The Lie bracket for vector fields is given by anti-symmetrising the pre-Lie product \eqref{preLieVect1}, i.e.,  for vector fields ${V, W}$ the bracket
\begin{equation}
\label{postLie3}
	[ V,W ]_{\mathcal \triangleright} := V \triangleright W - W \triangleright V
\end{equation}
satisfies the Jacobi identity. Notice that it is true in general that any pre-Lie algebra $(\mathfrak{g},  \triangleright )$ is Lie-admissible, i.e., we obtain a Lie algebra $(\overline{\mathfrak{g}}, [ \cdot, \cdot ]_{\triangleright} )$ by anti-symmetrising the pre-Lie product. 

Guin and Oudom showed in \cite{GO2008} how for any pre-Lie algebra $(\mathfrak{g},  \triangleright )$ the pre-Lie product can be extended to the symmetric algebra $\mathcal{S}(\mathfrak{g})$ over the vector space $\mathfrak{g}$. Choosing a basis $B=(b_i)_{i\in I}$ of $\mathfrak{g}$, 
the algebra $\mathcal{S}(\mathfrak{g})$ identifies with the algebra of polynomials generated by $B$. It is a graded, commutative, cocommutative Hopf algebra with  multiplication of polynomials as product, denoted $\cdot$, and  unshuffling coproduct, defined on monomials as follows:
$$
	\Delta_\sh (a_1\cdots a_n)
	=  \sum\limits_{{\{i_1,\ldots,i_k\}\coprod\{j_1,\ldots,j_{n-k}\}=[n]}}a_{i_1}\cdots a_{i_k}\otimes a_{j_1}\cdots a_{j_{n-k}},
$$
where the $a_i$ belong to $B$. 

One can then define a non-commutative, associative, unital product on $\mathcal{S}(\mathfrak{g})$, known as Grossman--Larson product  
\begin{equation}
\label{GLprod}
	X \ast Y := X^\prime \cdot (X^{\prime\prime} \triangleright Y), \quad X,Y \in \mathcal{S}(\mathfrak{g}),
\end{equation}
where the reduced Sweedler's notation is used: $\Delta_\sh(X) = X^\prime \otimes X^{\prime\prime}$. The $\triangleright$-product on the righthand side of \eqref{GLprod} is defined as an extension of the pre-Lie product from $(\mathfrak{g},  \triangleright )$ to $\mathcal{S}(\mathfrak{g})$ by
\begin{align}
	1 \triangleright X 	&:= X														\label{GO1}\\
	xY \triangleright z 	& := x \triangleright (Y \triangleright z) - (x \triangleright Y) \triangleright z 	\label{GO2}\\
	X \triangleright YZ 	&:= (X' \triangleright Y) \cdot (X'' \triangleright Z) 					\label{GO3},
\end{align}
for $X,Y,Z \in \mathcal{S}(\mathfrak{g})$ and $x,z \in \mathfrak{g}$. For instance, going back to \eqref{GLprod}, we find for elements $x,y,z \in \mathfrak{g} \subset  \mathcal{S}(\mathfrak{g})$
\begin{align}
	x \ast y 
		&= x \cdot y + x \triangleright y \label{GLp31} \\
	(x \cdot y) \ast z 
		&= x \cdot y \cdot z + x \cdot (y \triangleright z) +  y \cdot (x \triangleright z) 
				+  (x \cdot y) \triangleright z  \label{GLp32} \\
		&= x \cdot y \cdot z + x \cdot (y \triangleright z) +  y \cdot (x \triangleright z) 
				+  x \triangleright (y \triangleright z) - (x \triangleright  y) \triangleright z .
\end{align}
The identity 
$$
	(x \cdot y) \triangleright z = x \triangleright (y \triangleright z) - (x \triangleright  y) \triangleright z 
$$
follows from \eqref{GO2}.

With the new product \eqref{GLprod} and the unshuffling coproduct  $\mathcal{S}(\mathfrak{g})$ is a graded connected cocommutative Hopf algebra. In the particular case where $\mathfrak{g}$ is a free pre-Lie algebra it is known as the Grossman--Larson Hopf algebra. Moreover, Guin and Oudom showed that it is isomorphic to the enveloping algebra $\mathcal{U}(\overline{\mathfrak{g}})$, where we recall that $\overline{\mathfrak g}$ is the Lie algebra defined in terms of anti-symmetrising the pre-Lie product. 

Let us apply now these general results to the vector field $D_i=V_{i^{(1)}}$. We use freely the results of \cite{ChaPat2013} and refer to the article for proofs. Consider the free pre-Lie algebra $PL(x)$ over one generator $x$, denote $GL(x)$ its enveloping algebra as defined by the Guin--Oudom construction (the so-called Grossman--Larson Hopf algebra, see also \cite[Def.~6.3.1]{CP21}). We write $\nu$ for the evaluation map induced by $x \longmapsto V_{i^{(1)}}$. It maps an element in $PL(x)$ to a vector field in $\Rb^N$ (viewed as a differential operator) and more generally it maps an arbitrary element in $GL(x)$ to a differential operator. It follows from the original works of Grossman and Larson \cite{GL1,GL2} that $x^n/n!$ (where $x^n$ is the $n$-th power of $x$ in $GL(x)$ viewed as a polynomial algebra) maps to $D_{i^{(m)}}$. 

The series $\exp(x)=\sum_{n\in\Nb}x^n/n!$ is a group-like element in $\widehat{GL(x)}$, the completion of the graded Hopf algebra $GL(x)$. This is a general property in graded connected Hopf algebras: exponentials of primitive elements always form group-like elements in the right completions, see e.g.~\cite{CP21} for details. 

From Proposition 3.1 in \cite{ChaPat2013} and using Lemma \ref{explexpr}, the next proposition follows.

\begin{proposition}\label{propexplog}
The degree $n$ component $y_n$ of the logarithm $\log^\ast(\exp(x))$ in $\widehat{GL(x)}$ is an element of the completion of the free pre-Lie algebra $PL(x)$ given by
$$
	y_n=\sum_{k=1}^n \sum_{j_1+\cdots + j_k=n} \frac{(-1)^{k-1}}{k}x^{j_1} * \cdots  * x^{j_k}.
$$
Therefore, we also have
$$
	\nu(y_n)=\bar\mu((\bar i^{(n)})^0)=V_{i^{(n)}}.
$$ 
In particular, we see that $V_{i^{(n)}}$ is a vector field.
\end{proposition}

Recall now that the
pre-Lie Magnus expansion of an element $a$ in a pre-Lie algebra is obtained implicitly as (see \cite{EFMan2009})
\begin{equation}
\label{pre-Lie-Magnus}
	\Omega^\triangleright  (a) 
	= \sum_{n \ge 0} \frac{B_n}{n!} \ell^{(n)}_{\Omega^{\triangleright  }(a)\triangleright  }(a),
\end{equation}
{where $B_n$ is the $n$-th Bernoulli number.}
Here we define $\ell^{(n)}_{a \triangleright }(b):=\ell^{(n-1)}_{a\triangleright }(a \triangleright b)$ and $\ell^{(0)}_{a \triangleright }(b):=b$. This is a formal expansion that can be equivalently obtained by recursively computing the components $\Omega^\triangleright_n (a)$, $n>1$, using
\begin{equation}
\label{eq:MagnusRec}
	\Omega_n^\triangleright (a) 
	=  \sum_{k = 1}^{n-1}\dfrac{B_k}{k!}
	\sum_{\substack{i_1,i_2,\ldots,i_k \geq 1 \\ i_1 + i_2 + \cdots + i_k = n-1}} 
	\Omega^\triangleright_{i_1}(a) \triangleright \left(\Omega^\triangleright_{i_2}(a) 
	\triangleright \left(\cdots \triangleright\left(
	\Omega^\triangleright_{i_k}(a) \triangleright a\right)\right)\right).
\end{equation}
Here, $\Omega^\triangleright_n (a)$ refers to the degree $n$ component of the expansion, that is, the sum of terms in the formal expansion containing exactly $n$ times the element $a$. Expanding \eqref{pre-Lie-Magnus} in orders of the dummy parameter $h$, we obtain
\begin{align}
\label{plMagnusN3}
	\Omega^\triangleright  ( ha) 
	&= \sum_{n >0} \Omega^\triangleright_n(a) h^n \\
	&= ha - \frac{h^2}{2} a \triangleright a + \frac{h^3}{4} (a \triangleright a) \triangleright a 
							+ \frac{h^3}{12} a \triangleright (a \triangleright a) + \cdots .
\end{align}

It was shown in \cite{ChaPat2013} (see also \cite[Prop. 6.5.1]{CP21}) that the pre-Lie Magnus expansion can be obtained in $\widehat{GL(x)}$ as
\begin{equation}
\label{exprelation2}
	\Omega^{\triangleright  }(x) = \log^*\circ \exp(x).
\end{equation}

\begin{remark}
Identity \eqref{exprelation2} is well-known in the context of B-series and geometric integration theory under the name backward-error analysis \cite{HLW2006}. In a nutshell, the main idea is to determine a modified ODE such that its exact solution equals the numerical solution of the original problem. Indeed, if we consider the initial value problem $\dot{y}=V(y)$, $y(0)=y_0$, where $V$ is a vector field on $\mathbb{R}^{N}$, then its exact solution is given by the Grossman--Larson exponential, $y(t)=\exp^*(tV)$. Identity \eqref{exprelation2} is known to define the backward-error analysis of the explicit forward Euler method on $\mathbb{R}^{N}$, i.e., $\exp^*(\Omega^{\triangleright  }(hV))(y) = y + hV(y)$, where $h$ denotes the step size.  
\end{remark}

The following corollary, a direct consequence of Proposition \ref{propexplog} and Equation (\ref{exprelation2}), concludes the proof of the main Theorem \ref{thm:vectorfield}.

\begin{corollary}
The renormalised vector field $V_{i^{(n)}}$ is the $n$-th component of the pre-Lie Magnus expansion of {$V_i$}:
\begin{equation}
\label{eq:pLMagnusvectorfields}
	{V_{i^{(n)}}=\Omega_n^{\triangleright  }(V_i).}
\end{equation}
\end{corollary}

For example, with the pre-Lie product \eqref{preLieVect1} defined on vector fields, we find
\begin{align*}
	V_{i^{(2)}}
	&= -\frac{1}{2} V_i \triangleright  V_i\\
	V_{i^{(3)}}
	&= \frac{1}{4} (V_i \triangleright V_i) \triangleright V_i 
				+ \frac{1}{12} V_i \triangleright (V_i \triangleright V_i)\\
	V_{i^{(4)}}
	&= \frac{1}{24} V_i  \rhd \big((V_i  \rhd V_i ) \rhd V_i \big) 
			+ \frac{1}{24} (V_i  \rhd V_i ) \rhd (V_i  \rhd V_i ) \\
		&\hskip 10mm
			+ \frac{1}{8} \big((V_i  \rhd V_i ) \rhd V_i \big) \rhd V_i 
			+\frac{1}{24} \big(V_i  \rhd (V_i  \rhd V_i )\big) \rhd V_i. 
\end{align*}

\begin{remark}[Rooted trees 1]
\label{rmk:preLieMagtrees}
As the above examples as well as expansion \eqref{plMagnusN3} indicate, the expressions of the renormalised vector field $V_{i^{(n)}}$ in terms of the pre-Lie Magnus expansion \eqref{eq:pLMagnusvectorfields} are complicated. However, thanks to Cayley's \cite{Cayley1857}
 correspondence between non-planar rooted trees and specific vector fields, known as elementary differentials, the expansion can be expressed in a more compact way using rooted trees. It amounts to a pre-Lie morphism $F$ which goes from the free pre-Lie algebra in $\ell$ generators into the pre-Lie algebra (of vector fields) generated by  vector fields $V_i$, $ 0 \le i \le \ell$, on $\mathbb{R}^N$. Indeed, following \cite{ChapLiv2001}, the free pre-Lie algebra in $\ell$ 
generators $\{\Forest{[i]}\}_{i=0}^\ell$ consists of non-planar rooted trees with vertices decorated by $[\ell] = \{0,\ldots,\ell\}$ and with grafting $\curvearrowright$ as pre-Lie product. Recall that a non-planar rooted tree is a connected and simply connected graph with a distinguished vertex called the root. Vertices other than the root have exactly one outgoing edge and an arbitrary number of incoming edges. Edges are oriented towards the root vertex. A leaf is a vertex without any incoming edges. Recall the definition of the $B_+$-operator which maps a forest of rooted trees, $\tau_1 \cdots \tau_n$, to a rooted tree $\tau$, by grafting the roots of the $\tau_i$ to a single new root, that is, $B^+(\tau_1 \cdots \tau_n)=\tau$. For example
$$
	B^+(\Forest{[]})=\Forest{[[]]}
	\quad
	B^+(\Forest{[]}\Forest{[]})=\Forest{[[][]]}
	\quad
	B^+(\Forest{[]}\Forest{[[]]})=\Forest{[[][[]]]}
	\quad
	B^+(\Forest{[]}\Forest{[]}\Forest{[[]]}\Forest{[[]]}\Forest{[[][]]})=\Forest{[[][][[]][[]][[][]]]}.	
$$
Then, we define the symmetry factor $\sigma(\tau)$ of a rooted tree $\tau$ inductively, by $\sigma(\Forest{[]}) :=1$ and 
$$
	\sigma(B^+(\tau_1^{k_1} \cdots \tau_n^{k_n})) := k_1! \sigma(\tau_1) \cdots k_n!\sigma(\tau_n).
$$
Let $T$ denote the set of non-planar rooted trees and $\mathcal{T}_{[\ell]}$ denotes the vector space spanned by rooted trees decorated by $[\ell]$. A tree $\tau$ with all its vertices decorated by $i \in [\ell]$ is denoted $\tau^{(i)}$. Starting from $F(\Forest{[i]})=V_i$, we then see that 
\begin{align*}
	F(- \frac{1}{2} \Forest{[i[i]]}) 
	&= F(- \frac{1}{2} \Forest{[i]} \curvearrowright \Forest{[i]} ) 
	   = -\frac{1}{2} V_i \triangleright V_i = \Omega_2^\triangleright (V_i) \\
	F(\frac{1}{3}\, \Forest{[i[i[i]]]} + \frac{1}{12}\, \Forest{[i[i][i]]})   
	&=
	F(\frac{1}{4} (\Forest{[i]}  \curvearrowright\Forest{[i]} ) \curvearrowright \Forest{[i]}  
					+ \frac{1}{12} \Forest{[i]} \curvearrowright (\Forest{[i]} \curvearrowright \Forest{[i]} ))\\	
	&=  \frac{1}{4} (V_i \triangleright V_i) \triangleright V_i 
					+ \frac{1}{12} V_i \triangleright (V_i \triangleright V_i)
					= \Omega^\triangleright_3 (V_i) \\
	F(\frac{1}{4} \Forest{[i[i[i[i]]]]} + \frac{1}{12} \Forest{[i[i[i][i]]]} + \frac{1}{12}\Forest{[i[i][i[i]]]})
	&=  \frac{1}{8} \big(V_i \rhd (V_i \rhd V_i)\big) \rhd V_i  
			+ \frac{1}{24} (V_i \rhd V_i) \rhd (V_i \rhd V_i) \\
			 &\quad
			 + \frac{1}{24}  V_i \rhd \big((V_i \rhd V_i) \rhd V_i\big)
			+ \frac{1}{24} V_i \rhd \big((V_i \rhd V_i) \rhd V_i\big)\\
	&=\Omega^\triangleright_4 (V_i) .
\end{align*}
In the free pre-Lie algebra in $\ell$ generators, the pre-Lie Magnus series \eqref{pre-Lie-Magnus} can be expressed as an expansion in non-planar rooted trees 
\begin{align}
\label{eq:MagnusTree}
	\Omega^\triangleright (\Forest{[i]}) 
	&= \sum_{\tau \in T} c_\tau \tau^{(i)}\\
	&= \Forest{[i]} 
		- \frac{1}{2} \Forest{[i[i]]}
		+ \frac{1}{3}\, \Forest{[i[i[i]]]} + \frac{1}{12}\, \Forest{[i[i][i]]} 
		+ \frac{1}{4} \Forest{[i[i[i[i]]]]} + \frac{1}{12} \Forest{[i[i[i][i]]]} + \frac{1}{12}\Forest{[i[i][i[i]]]} + \cdots \nonumber
\end{align}
such that 
\begin{align}
\label{eq:elemdiffMagnus}
	F(\Omega^\triangleright (\Forest{[i]}) ) = \Omega^\triangleright (V_i).
\end{align}
The coefficients $c_\tau$ in \eqref{eq:MagnusTree} can be computed recursively
\begin{equation}
\label{eq:MagnusTreecoeff}
	c_\tau:=\frac{\omega(\tau)}{\sigma(\tau)},
\end{equation}
where the function $\omega(\tau)$ can be computed in several ways: we refer to \cite{CelPat2023} and references therein. 
 
\begin{equation}
\label{omegacoefftable1}
\begin{tabular} {c|c|c|c|c|c|c|c|c|c|c|c|c}
	$\tau$ & \scalebox{0.5}{\Forest{[]}} & \scalebox{0.5}{\Forest{[[]]}} & \scalebox{0.5}{\Forest{[[[]]]}} 
	& \scalebox{0.5}{\Forest{[[][]]}} & \scalebox{0.5}{\Forest{[[[][]]]}} & \scalebox{0.5}{\Forest{[[[[]]]]}} 
	& \scalebox{0.5}{\Forest{[[][[]]]}} & \scalebox{0.5}{\Forest{[[][][]]}} & \scalebox{0.5}{\Forest{[[][][][]]}} \\[0.3cm]
\hline	
	$\omega$ & $1$ & $\frac{1}{2}$ & $\frac{1}{3}$ & $\frac{1}{6}$ & $ \frac{1}{6}$ & $\frac{1}{4}$ 
	& $\frac{1}{12}$ & $0$ & $-\frac{1}{30}$ \\ \hline
	$\sigma$ &  1 & 1 & 1 & 2 & 2 & 1 & 1 & 3! & 4!
\end{tabular} 
\end{equation}
\begin{equation}
\label{omegacoefftable2}
\begin{tabular} {c|c|c|c|c|c|c|c|c|c|c|c|c}
	$\tau$ 	& \scalebox{0.5}{\Forest{[[[][][]]]}} & \scalebox{0.5}{\Forest{[[[]][[]]]}} & \scalebox{0.5}{\Forest{[[][[][]]]}} 
	& \scalebox{0.5}{\Forest{[[[[[]]]]]}} & \scalebox{0.5}{\Forest{[[][[[]]]]}} & \scalebox{0.5}{\Forest{[[[[][]]]]}} 
	& \scalebox{0.5}{\Forest{[[[][[]]]]}} 
	& \scalebox{0.5}{\Forest{[[][][[]]]}} \\[0.3cm]
\hline	
	$\omega$ & $\frac{1}{30}$ & $\frac{1}{30}$ & $\frac{1}{60}$ & $\frac{1}{5}$ 	
	& $\frac{1}{20}$ 	& $\frac{3}{20}$ 	& $\frac{1}{10}$ &
	$\frac{-1}{60}$ \\ \hline
	$\sigma$ &  3! & 2 & 2 & 1 & 1 & 2 & 1 & 2 
\end{tabular}
\end{equation}
\end{remark}

\begin{remark}[Renormalised SDE]
The renormalised expression of the identity matrix of $\Rb \la \Ab\ra$ using the alphabet $\Ab^0$ (instead of $\Ab$), the renormalisation of the iterated integrals, i.e., the $J_w$ terms used in place of the iterated It\^o integrals $I_w$, together with the renormalisation of the vector fields and differential operators, i.e., the $V_w$ terms instead of the $D_w$, allows to formally rewrite the initial stochastic differential equation in the form
\begin{equation}
\label{renSDE}
	Y_t = V_0  {\mathrm d}t 
%		+ \sum_{i=1}^d V_i \circ\! dW_i 
		+ \sum_{i=1}^d \sum_{k=1}^2 \Omega_k^\triangleright (V_i) \circ\! {\mathrm d}[W_i]^{(k)} 
		+ \sum_{i=d+1}^\ell \sum_{m>0} \Omega_m^\triangleright (V_i) \circ\! {\mathrm d}[J_i]^{(m)} ,
\end{equation}
where the symbol $\Omega_k^\triangleright (V_i)$ stands for the $k$-order component of the pre-Lie Magnus expansion (see above). The symbols $\circ {\mathrm d}[W_i]^{(k)}$ and $\circ {\mathrm d}[J_i]^{(m)}$ refer to the fact that the corresponding iterated stochastic integrals are computed with the processes $J_w$. It will be interesting to explore further the properties of these equations.
\end{remark}

\begin{remark}[Rooted trees 2]
Going back to the explicit expansion of the pre-Lie Magnus expansion using trees \eqref{eq:MagnusTree} and formula \eqref{eq:MagnusTreecoeff} for the coefficients, we can express \eqref{renSDE} compactly as a tree expansion
$$
	Y_t = V_0  {\mathrm d}t 
	+ \sum_{i=1}^d \sum_{{\tau \in T} \atop {1 \le |\tau| \le 2}} c_\tau\mathcal{F}_{V_i}(\tau) \circ\! 
{\mathrm d}[W_i]^{(|\tau|)}
	+ \sum_{i=d+1}^\ell \sum_{\tau \in T} c_\tau\mathcal{F}_{V_i}(\tau)\circ\! {\mathrm d}[J_i]^{(|\tau|)}  .
$$
Here, the set of rooted trees is denoted $T$ and $|\tau|$ is the degree, i.e., the number of vertices of $\tau \in T$. The elementary differential $\mathcal{F}$ is defined for any tree $\tau=B^+[\tau_1 \cdots \tau_k]$ by
$$
	\mathcal{F}_{V_i}(\tau)
	:= V_{i}^{(k)}(\mathcal{F}_{V_{i}}(\tau_1),\ldots, \mathcal{F}_{V_{i}}(\tau_k)),
$$
\end{remark}
where $V_{i}^{(k)}$ is the $k$-th derivative of $V_i$, i.e., a $k$-linear mapping. As examples, we consider 
\begin{align*} 
	\mathcal{F}_{V_i} (\Forest{[[]]})
	&=V_i^{(1)} (V_i)
		= \sum^N_{j,k=1}V_i^j (\partial_j V_i^k)\partial_k \\
	\mathcal{F}_{V_i}(\Forest{[[][]]})
	&=V_i^{(2)}(V_i,V_i)
		= \sum^N_{j,k,m=1} V_i^kV_i^m(\partial_k\partial_mV_i^j) \partial_j\\
	\mathcal{F}_{V_i}(\Forest{[[[]]]})
	&=V_i^{(1)}(V_i^{(1)}(V_i)) 
		= \sum^N_{j,k,m=1} V_i^j (\partial_j V_i^k)(\partial_kV_i^m)\partial_m.
\end{align*}

%%%%%%%%%%%%%%%%%%%%%%%%%%%%%%%%%%%%
%%%%%%%%%%%%%%%%%%%%%%%%%%%%%%%%%%%%

\section{Chen--Strichartz formula} 
\label{sec:Chen_Strichartz} 

The renormalised representation of the flowmap $\varphi_t=\sum_w J_w(t) V_w$ allows us to extend
the results in \cite{exp_Lie_series}, Theorem 4.8 and Corollary 4.11 from stochastic differential equations driven by continuous semimartingales to L{\'e}vy-driven stochastic equations.

\begin{theorem}[Chen--Strichartz series] 
\label{Th:Chen-Strichartz} 
The logarithm of  the flowmap has the following series representation
\begin{enumerate}
\begin{align*}
	\log\varphi_t 
& \,= \log\biggl(\sum\limits_{w\in\Ab^\ast} J_w(t) V_w\biggr)
\\
 &\,   = \sum\limits_{w\in\Ab^\ast}J_w(t)\biggl(\sum_{\sigma\in{\mathbb S}(|w|)}
\frac{1}{|\sigma|} c_\sigma 
V_{[\sigma (w)]_{\mathrm L}}\biggr),
\end{align*} 
\end{enumerate}
where 
$$
	c_\sigma:=\frac{(-1)^{d(\sigma)}}{|\sigma|}
	\biggl(\begin{array}{c}
	|\sigma|-1\\
	d(\sigma)
		\end{array}
	\biggr),
$$
and where $d(\sigma)$ denotes the number of descents in the permutation $\sigma \in {\mathbb S}_{|w|}$ and $|\sigma|:=|w|$. In particular, it can be represented as
$$\log\varphi_t = \sum\limits_{w\in\Ab^\ast} J_w(t) V_{\psi(w)},$$
where
$\psi(w):=\sum\limits_{\sigma\in{\mathbb S}(|w|)} 
\frac{1}{|\sigma|}c_\sigma [\sigma (w)]_\cL$ is a Lie series.
Here $[w]_{\mathrm L}$ denotes the left-to-right Lie-bracketing of its letters, i.e., we have 
$
	[a_1\cdots a_n]_{\mathrm L}
	=[a_1, [a_2, \ldots [a_{n-1}, a_n]_{\mathrm L}]_{\mathrm L}\ldots ]_{\mathrm L}.
$
\end{theorem}

The proof of the Theorem is of purely combinatorial nature and depends only of the fact that the $J_w$ multiply according to the shuffle product of words, whereas the $V_w$ multiply according to concatenation. The formula is thus the same as the classical Chen--Strichartz formula and we do not repeat the proof here, it can be found in \cite[Sect. 3.2]{Reutenauer93} or in \cite{exp_Lie_series} that uses a Hopf algebraic approach, conventions and notation closer to the ones in the present article.

It is worthwhile noting that the proof in the latter reference also establishes the following representations for 
$\log\varphi_t$ as a series in the vector field compositions $V_w$ and the iterated integrals $J_w(t)$, respectively:
\begin{align}
\log\varphi_t 
     &= \sum\limits_{w\in\Ab^\ast} J_{\log^\sh(w)} (t)V_{w}\\
& = \sum\limits_{w\in\Ab^\ast} J_w (t)
c_\sigma V_{\sigma(w)},
\end{align}  
where $\log^\sh(w)=\sum_{\sigma \in {\mathbb S}_{|w|}} c_\sigma \sigma^{-1}(w)$.

It will often be more convenient to express the Chen-Strichartz series in terms of the iterated It{\^o} integrals $I_w$. 
For consistency with the treatment of continuous semi-martingales in \cite{exp_Lie_series}, we denote  
$\exp_{\mathrm H}^\dag$ the
 transformation obtained in Lemma
\ref{lemma:exp_dual} describing the basis change from $\overline{A}^\ast$ to
$(\overline{A}^0)^\ast$ (the notation is a reference to Hoffman's exponential isomorphism between shuffle and quasi-shuffle algebras). Thus,
for $\overline{a}\in \overline{A}$ in 
$$
	\overline a
	= \sum_{n\ge 1}\sum_{[a_1, \ldots, a_n]=a} 
		\frac{1}{n!} \overline a_1^0 \cdots \overline a_n^0=:\exp_{\mathrm H}^\dag(\overline a^0).
$$
For a word $\overline w=\overline a_1\dots \overline a_n\in\overline\Ab^\ast$, its expression denoted $\exp_\mathrm{H}^\dag(\overline w^0)$ in the basis $(\overline\Ab^0)^\ast$ is obtained by
replacing each of the letters in $\overline w$, say $\overline{i}^{(m)}$, with
\[
	\exp_{\mathrm H}^\dag\bigl((\overline i^{(m)})^0\bigr)=\sum_{n\ge 1}\frac{1}{n!}
	\sum_{j_1+\ldots +j_n=m} (\overline i^{(j_1)})^0\ldots (\overline i^{(j_n)})^0.
\]
That is,
$$\overline w=:\exp_{\mathrm H}^\dag(\overline w^0):=\exp_{\mathrm H}^\dag(\overline a_1^0)\cdots \exp_{\mathrm H}^\dag(\overline a_n^0).$$

\begin{corollary}[It{\^o} Lie series] \label{cor:Ito_Lie_series}
In terms of the It{\^o} integrals $I_w$ the Chen--Strichartz series is given by
\begin{align*}
\log\bigl(\varphi_t\bigr) &  = 
\log\biggl(\sum\limits_{w\in\Ab^\ast} I_w(t)D_w\biggr) 
  = \sum\limits_{w\in\Ab^\ast} I_{w}(t) \biggl(\sum_{\sigma\in\Sb(|w|)}
\frac{1}{|\sigma|} c_\sigma 
V_{[\exp_{\mathrm H}^\dag (\sigma (w))]_{\cL}}\biggr).
\end{align*}
\end{corollary}

\begin{proof} For notational simplicity, we abbreviate $\sum_{w\in\Ab^\ast}$ to $\sum_{w}$, and we omit the dependency on $t$ in the proof. It follows from Theorem \ref{Th:Chen-Strichartz} that $\log\varphi$ as a series in the It{\^o} integrals
$I_w$ is a Lie series. 
We recall the representations 
$$
\mu(w):=I_w, \quad \bar{\mu}(\overline{w}):=D_w, \quad
\mu(w^0)=:J_w, \quad \bar{\mu}(\overline{w}^0)=:V_w.
$$  
Hence we have 
\begin{align*}
\log\bigl(\varphi\bigr) &\, = \, 
\log\biggl(\sum_w I_wD_w\biggr)\\
&\,=\,
 \log\biggl(\sum_w J_w V_w\biggr)\\
& \, = \, \sum_w\sum_{\sigma\in{\mathbb S}(|w|)}
 c_\sigma J_w 
V_{\sigma (w)}\\
& \, = \, \mu\otimes\bar{\mu}
\biggl(\sum_{w^0} \sum_{\sigma \in {\mathbb S}(|w^0|)}
 w^0\otimes  c_\sigma \sigma(\overline w^0)\biggr)\\
& \, = \, \mu\otimes\bar{\mu}
\biggl(\sum_{w} \sum_{\sigma \in {\mathbb S}(|w|)}\biggl(\sum\limits_{i_1+\ldots+i_l=|w|}\frac{1}{i_1!\cdots i_l!}(i_1,\cdots,i_l)\circ
w\biggr)\otimes c_\sigma  \sigma(\overline w^0)\biggr)\\
& \, = \, \mu\otimes\bar{\mu}
\biggl(\sum_{w} \sum_{\sigma \in {\mathbb S}(|w|)}
 w\otimes
c_\sigma \exp_{\mathrm H}^\dag (\sigma(\overline w^0))\biggr)\\
& \, = \, \mu\otimes\bar{\mu}
\biggl(\sum_{w} \sum_{\sigma \in {\mathbb S}(|w|)}
 w\otimes \frac{1}{|\sigma|} 
c_\sigma [\exp_{\mathrm H}^\dag (\sigma(\overline w^0))]_{\cL}\biggr)\\
& \, = \, 
\sum_{w} I_w\biggl(\sum_{\sigma \in {\mathbb S}(|w|)}
\frac{1}{|\sigma|} c_\sigma V_{[\exp_{\mathrm H}^\dag (\sigma(w))]_{\cL}}\biggr),
\end{align*} 
as required.
Here we have used for the fifth identity the formula expressing the change
from $J_w$ to $I_w$, see Lemma \ref{lemma:exp}. The sixth identity can be obtained by a direct computation, that we omit. It also follows more abstractly from a duality argument: the 
basis change from $J_w$ to $I_w$ dualises to the basis change between $D_w$ and $V_w$, see Lemma \ref{lemma:exp_dual}. The seventh identity follows from the properties of the Dynkin operator $[\ ]_{\cL}$, see \cite{Reutenauer93}, together with the fact that the right hand side term of 
$\log\bigl(\varphi\bigr)$ is a Lie series, see Theorem \ref{Th:Chen-Strichartz}.
\end{proof}

\begin{remark} As noted above, 
Theorem \ref{Th:Chen-Strichartz} and Corollary \ref{cor:Ito_Lie_series} show that for It{\^o} stochastic differential equations driven by L{\'e}vy processes the logarithm of the flowmap $\varphi_t$ is a Lie series. 

Key in the proof of the Chen--Strichartz formula is the change of basis representing the change from the It{\^o} integrals $I_w$ to the processes $J_w$, that satisfy the classical integration by parts formula, and on the dual side, the change of basis respresenting the change from the differential operators $D_w$
to the composition $V_w$ of the renormalised vector fields. 
Alternatively, we could have deduced these formulas using the Hoffman isomorphisms between quasi-shuffle Hopf algebra and  shuffle Hopf algebra and their adjoints, see \cite{Hoffman2000}.

We also observe how the transformation between the It{\^o} integrals $I_w$ and the processes $J_w$ involves
not only the L{\'e}vy processes driving the stochastic differential
equation but also all power bracket process. This is reminiscent
of the change from It{\^o} integral to Fisk--Stratonovich  integral
for continuous processes.
\end{remark}

%%%%%%%%%%%%%%%%%%%%%%%%%%%%%%%%%%%%
%%%%%%%%%%%%%%%%%%%%%%%%%%%%%%%%%%%%

\section{Concluding Remarks}\label{sec:conclusion}

The exponential Lie series and Chen--Strichartz formula are well established for the design 
of numerical integration schemes in deterministic
settings and for Wiener-driven stochastic differential
equations, e.g., in the Wiener setting Lord, Malham \& Wiese
\cite{LMW2008} design numerical schemes for the strong solution of Wiener-driven stochastic differential equations based on the exponential Lie series, and Malham \& Wiese
 \cite{Lie_Group_paper} design Lie group methods for solutions that evolve on homogeneous manifolds. In this paper 
we have shown that the logarithm of the flowmap for L{\'e}vy-driven stochastic differential equations is a Lie series. We have furthermore derived an explicit
formula for the terms of the Lie series. These results expand our previous
results establishing the exponential Lie series for continuous semimartingales, see \cite{exp_Lie_series}. A novel and key argument was to define a change of basis such that the new corresponding It{\^o} processes, obtained through this basis change, satisfy the classical integration by parts formula. 
We showed the corresponding differential operators obtained through the basis change are in fact vector fields. Moreover, we identified them explicitly as the components of the pre-Lie
Magnus expansion generated by the original vector fields governing our stochastic system.  

\section*{Acknowledgment}
AW acknowledges support by the EPSRC, grant number EP/Y033248/1. KEF acknowledges support by the Research Council of Norway, project 302831 'Computational Dynamics and Stochastics in Manifolds (CODYSMA)'.

%%%%%%%%%%%%%%%%%%%%%%%%%%%%%%%%%%%%
%%%%%%%%%%%%%%%%%%%%%%%%%%%%%%%%%%%%
%%%%%%%%%%%%%%%%%%%%%%%%%%%%%%%%%%%%

%___________ Bibliography ___________________________
%
\end{document}